\documentclass[11pt]{article}
\usepackage[margin=1in]{geometry}
\pdfoutput=1
\usepackage{graphicx}
\usepackage{authblk}
\usepackage{amsmath}
\usepackage{amsfonts}
\usepackage{amsthm}
\usepackage{amssymb,xspace,float,mathtools}
\usepackage[T1]{fontenc}
\usepackage[utf8]{inputenc}
\usepackage[english]{babel}
\usepackage[round]{natbib}
\usepackage{bm}
\usepackage{algorithm}
\usepackage{algpseudocode}
\usepackage{booktabs}
\usepackage{caption}
\usepackage{hyperref}
\hypersetup{
    colorlinks=true,
    linkcolor=black,
    filecolor=magenta,
    urlcolor=blue,
    pdftitle={Liquidity Pools as Mean Field Games},
    }

\newtheorem{thm}{Theorem}[section]

\newtheorem{prop}[thm]{Proposition}
\newtheorem{cor}[thm]{Corollary}

\newtheorem*{obs*}{Remark}
\newtheorem{remark}[thm]{Remark}
\newtheorem*{remark*}{Remark}

\theoremstyle{definition}
\newtheorem{Def}[thm]{Definition}

\newcommand{\keywords}[1]{\textbf{\textit{Keywords---}} #1}

\begin{document}
\title{A New Framework for Modelling Liquidity Pools as Mean Field Games}

\author[1]{Agustín Muñoz González}

\author[2,3]{\\ Rafael Orive-Illera}

\author[1,4]{\\ Juan I. Sequeira}

\affil[1]{Departamento de Matem\'aticas, Facultad de Ciencias Exactas y Naturales, Universidad de Buenos Aires, Buenos Aires, Argentina}

\affil[2]{Departamento de Matem\'aticas, Facultad de Ciencias, Universidad Aut\'onoma de Madrid, Madrid, España}

\affil[3]{Instituto de Ciencias Matemáticas, CSIC-UAM-UCM-UC3M,  Nicolás Cabrera 13-15, Madrid, 28049, España}

\affil[4]{IMAS-CONICET, Buenos Aires, Argentina}

\date{April 2026}

\maketitle

\begin{abstract}
In this work, we present an application of the probabilistic weak formulation of mean field games (MFG) for modeling liquidity pools in a constant product automated market maker (AMM) protocol in the context of decentralized finance. Our work extends one of the most conventional applications of MFG, which is the price impact model in an order book, by incorporating an AMM instead of a traditional order book. The key structural difference is that in the AMM setting, the price is determined by the pool's reserves through a nonlinear mechanism, replacing the linear price-impact function used in classical models. Through our approach, we establish the existence of solutions to the Mean Field Game and, additionally, the existence of approximate Nash equilibria for the finite-player game. We complement the theoretical results with a comprehensive numerical study that validates the equilibrium structure: stability under perturbations, the $\varepsilon$-Nash property via unilateral deviations, finite-player convergence at propagation-of-chaos rates, and sensitivity to cost parameters and incentive targets. These results offer a new game-theoretic perspective for representing strategic behavior in AMM-based liquidity pools and open promising opportunities for future research in this emerging field.
\end{abstract}

\vspace{10pt}

\section{Introduction}
In recent years, we have witnessed growing excitement around cryptocurrencies and \emph{decentralized exchanges} (DEX) based on blockchain technology. A DEX is a type of exchange that specializes in peer-to-peer cryptocurrency and digital asset transactions. Unlike centralized exchanges (CEXs), DEXs do not require a trusted third party, or intermediary, to facilitate the exchange of cryptoassets. 

DEX trading volumes have shown remarkable growth since 2020, as illustrated in Figure \ref{Fig: Trading Vol}, reaching an early peak of over \$200 billion in May 2021 during the cryptocurrency bull run. While volumes fluctuated in the subsequent years, they remained at historically high levels, underscoring the resilience and continued adoption of decentralized finance. A new record was set in January 2025, with volumes surpassing \$400 billion, driven by strong activity across a broader range of platforms. In recent months, the market has once again surpassed previous volumes, with Uniswap, PancakeSwap, and a growing set of other DEXs contributing significantly to overall liquidity.

\begin{figure}[H]
    \centering
    \includegraphics[width=16cm]{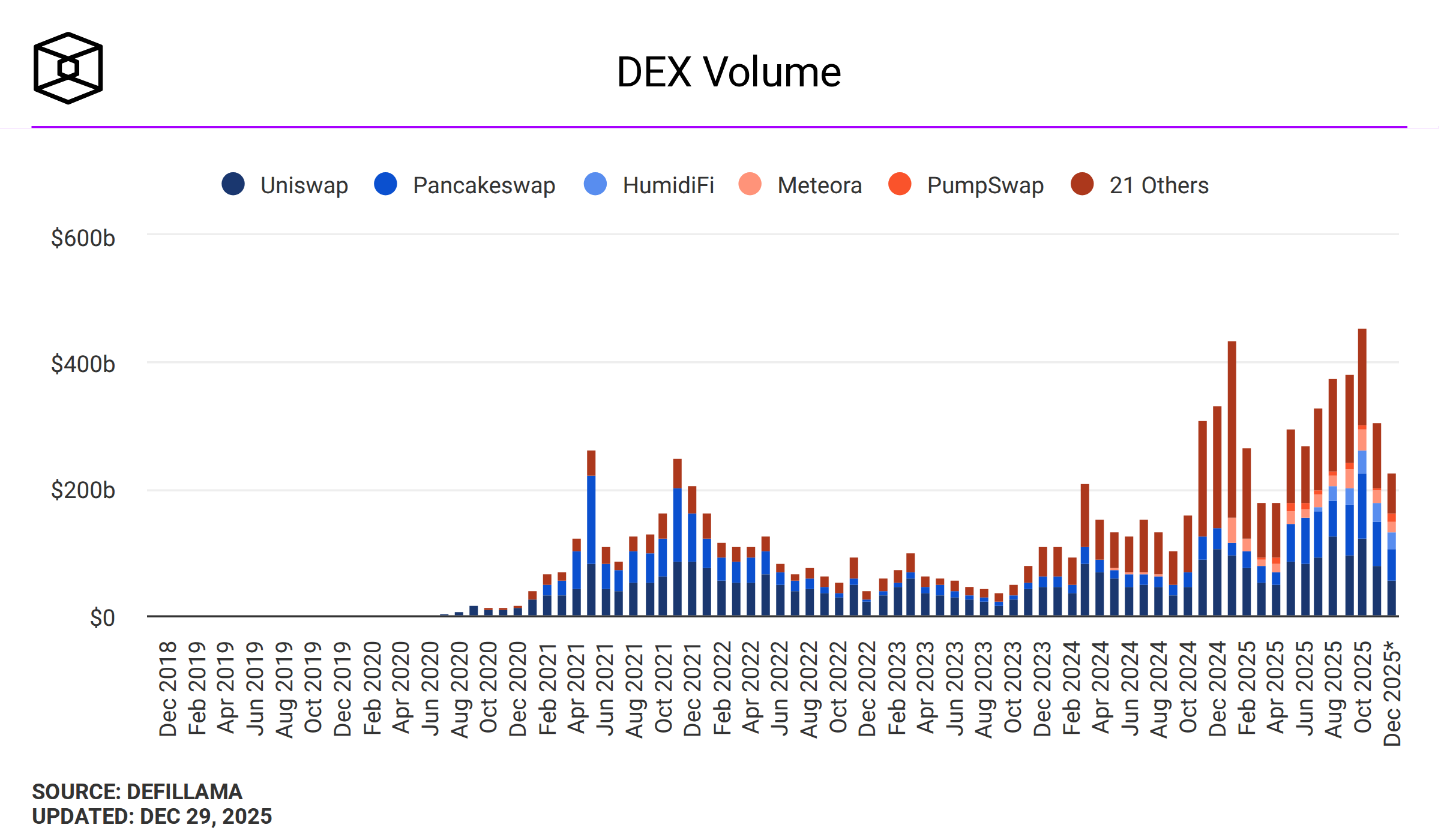}
    {\small\\ Sources: The Block, DefiLlama (\href{https://www.theblock.co/data/decentralized-finance/dex-non-custodial}{The Block}, \href{https://defillama.com/dexs}{DefiLlama})}
    \caption{Monthly trading volume on DEXs.}
    \label{Fig: Trading Vol}
\end{figure}

 In DEX markets, price formation is carried out through algorithms known as \emph{automated marked makers} (AMMs), see
 \citep{bayraktar2024DEX, mohan2021amm, xu2022SoK}. AMMs match trading orders and determine asset prices through a mathematical function called the "exchange function".  Unlike traditional market-making protocols, such as order books, AMMs do not require the physical presence of market makers or intermediaries for order execution and price setting. In an automated market, traders do not trade directly with each other but transact with assets stored in an exchange, known as "liquidity pool", see \citep{he2024AMM}. These liquidity pools are supplied by liquidity providers who wish to facilitate traders' operations.

Suppose that in a liquidity pool, exchanges of assets $X$ and $Y$ occur with initial reserves $x$ and $y$, respectively. When a trader buys an amount $\delta$ of tokens $Y$, paying a value of $p\delta$ in tokens $X$, this affects the pool's reserves as follows: the value of token $Y$ in the pool decreases ($y_0 \to y = y_0 - \delta$), while the adjusted value in tokens $X$ increases ($x_0 \to x = x_0 + p\delta$). These changes in the pool's reserves cause a transition from $(x_0, y_0)$ to $(x, y)$. In the particular case of a constant product AMM, it is required that the product of the liquidity pool's reserves, excluding trading fees, remains constant, i.e., $k = xy = x_0y_0$ for some fixed value $k$. This equation determines the execution price $p$ for this transaction \citep{angeris2019analysis,mohan2021amm,xu2022SoK}.

Mean field games (MFG) are a theoretical framework that describes strategic interaction in systems with a large number of agents who indirectly influence each other \citep{guanxing2023MFL,gueant2011applications,lachapelle2011congestion}. The MFG formalism focuses on studying interactions in scenarios where the actions of an individual agent have a negligible impact on the whole, but the collective significantly influences each agent's outcome. This approach has found applications in diverse areas such as economics, engineering, biology, and, more recently, in the analysis of decentralized financial markets. MFGs allow modeling and predicting complex dynamics within these systems, facilitating the search for equilibria that represent stable states of the system under certain decision policies.

Typically, the equilibrium in MFGs is determined by the distribution of the agents' states. However, in some economic applications, the cost faced by each agent depends on the distribution of the controls rather than the states, a framework known as "mean field games of controls" \citep{gomes2016extendedmfg, gomes2014solutionextendedmfg, graber2020variational}. Our approach aligns with this framework, as the price discovery in the AMM is inherently linked to the aggregate actions of the traders, affecting the pool's reserves and thus the price.

In a financial market, especially in the case of cryptocurrencies and digital assets, investors' and traders' decision-making can be highly strategic and influenced by a series of factors. The use of mean field game theory is justified in this context due to its ability to provide a robust framework that allows understanding how market participants interact in a decentralized environment and how their collective actions shape the overall market behavior. It was introduced independently by Lasry and Lions and by Huang, Malhamé, and Caines around 2006 \citep{huang2006largegames, lasry2007mfg}. Since then, MFG has been extended to various domains, including price formation models \citep{gomes2021mfgpriceformation, lions2016priceformation}. In traditional financial markets, these models typically treat demand as a function of price or vice versa. In contrast, our model leverages the unique structure of AMMs, where the price is a function of the liquidity reserves. This allows us to account for the direct impact of traders' actions on price formation.

Our research builds upon the price impact model used in centralized markets, where order books typically serve as the exchange mechanism \citep{carmona2015weakmfgformulation}. We adapt this model to address the control problem each trader faces when seeking to optimize their inventory while accounting for associated costs. However, unlike the traditional setup, our approach captures the price formation arising directly from the interaction of agents within the AMM, leveraging the unique characteristics of decentralized liquidity pools. We establish the existence of solutions to the mean field game and demonstrate the existence of approximate Nash equilibria, which provide a realistic extension to the classical Nash equilibrium concept by accommodating situations where exact equilibria are impractical to achieve. In this setting, an approximate Nash equilibrium is a state where no trader can significantly improve their outcome through a unilateral change in strategy, within a specified tolerance.

This paper advances the literature on AMMs and the use of bonding curves in DeFi by introducing a novel framework for applying mean field game theory to model liquidity pools. Our approach addresses the decentralized nature of AMMs by modeling the actions of "swappers" in a Uniswap v2-style liquidity pool without relying on probabilistic distributions for trader behavior. This direct modeling approach allows us to capture the strategic interactions among agents more accurately and provides new insights into equilibrium dynamics and decision-making processes in decentralized markets.

We conclude this introduction with the structure of the paper. In Section 2 we detail the inventory problem faced by a trader operating in a liquidity pool. We explore the specific dynamics and constraints that influence these operators' decision-making and conclude the section with a model for the price of the ETH asset in the liquidity pool.
In Section 3, we apply the model results to the token price in the liquidity pool, analyzing its behavior and evaluating its impact in the context of the mean field game. Here, we  examine how market participants' strategies and the dynamics of the liquidity pool influence price formation and decision-making. 
Finally, in Section 4, we present the conclusions derived from our study, summarizing the key findings, highlighting potential areas for improvement, and proposing future research that could expand and enrich the field of study. Our goal is to contribute to the understanding of cryptocurrency market dynamics and provide valuable insights for investors, traders, and academics interested in this exciting and dynamic field.

\section{Model Formulation}
\label{sec:Model Formulation}
We consider $N$ traders who trade their two tokens ETH and USDC in a liquidity pool. Let $X_t^i$ and $Y_t^i$, $i=1\dots,N$, be the inventories of ETH and USDC, respectively, for the $i$-th trader at time $t$. The dynamics of the ETH inventory for the $i$-th trader are defined as
\begin{equation}\label{dynamics of X}
  dX_t^i = \alpha_t^i dt + \sigma^i dW_t^i,  
\end{equation}
where $\alpha_t^i : [0,T] \to \mathbb{R}$ represents the trading rate, these will be the controls; $\sigma^i$ represent the respective volatilities, which for simplicity we assume are independent of $i$; and $\sigma^i dW_t^i$ represents a random change in the trader's inventory, associated with exogenous market events or inherent wallet usage, where $W_t^i$ are $N$ uncorrelated Wiener processes.

On the other hand, the dynamics of the USDC inventory are
$$ dY_t^i = -(\alpha_t^i P_t + c_p(\alpha_t^i)) dt, $$
where $P_t$ is the price at time $t$ of ETH in USDC units in the liquidity pool. The function $c_p(\alpha)$ models the cost for the trading rate $\alpha$ of trading in the pool. For this first model, we will consider a liquidity pool given by a constant product AMM without transaction costs, i.e., $c_p \equiv 0$. 

The price of ETH in USDC at time $P_t$ is derived from the ratio 
$$ P_t = \frac{Y_t}{X_t} $$
of the two asset balances $X_t$ and $Y_t$, the balances in the pool of ETH and USDC, respectively. Furthermore, in a typical constant product liquidity pool (like Uniswap), the balances satisfy the equation $X_t Y_t = k$ for all $t$, with $k>0$ constant, and we have
$$ P_t = \frac{k}{X_t^2}. $$

According to the dynamics of ETH inventories for $N$ agents with $N \in \mathbb{N}$, we model the balance of ETH in the pool as follows
$$ X_t := X_0 - \frac{1}{N} \sum_{i=1}^N \int_0^t \alpha_s^i ds, $$
where $X_0$ is the number of tokens provided by the liquidity provider at an initial time.
A detail to consider is that in a continuous-time trading model with multiple agents, it is unclear how simultaneous trades should be handled. More realistic are continuous-time discrete trading models. In the latter, it is reasonable to assume that agents never trade tokens simultaneously, given that there is a continuum of times to choose from. In general, considering the average is a common practice in Mean Field Games theory, since when $N$ tends to infinity, the average tends to the empirical distribution of the controls. 

\begin{remark*}
    In \ref{sup: S} after the proof of assumption \textbf{\hyperref[sup: S.4]{(S.4)}}, we will have that $X_t \neq 0$ for all $t \in [0,T]$. Therefore, the pool maintains at least a minimal fraction of tokens at all times
\begin{equation}\label{Xpositive}
    X_t \geq \epsilon_0, \quad \forall t\in[0,T], 
\end{equation}
for a certain $\epsilon_0>0$.
\end{remark*}

Finally, the price of ETH is modeled as a martingale plus a drift that represents the permanent price impact
$$ P_t := \frac{k}{\displaystyle \left(X_0 - \int_0^t \frac{1}{N} \sum_{i=1}^N \alpha_s^i ds \right)^2} + \sigma^0 W_t^0, $$
where $W_t^0$ is an uncorrelated Wiener process that models the inherent risks of the liquidity pool (slippage, illiquidity, oracle failures, exploits, etc.), and $\sigma^0$ is an associated volatility.

\begin{remark*}
    While the price process $(P_t)_{0\leq t\leq T}$ could, in theory, become negative due to the combined effect of trader actions and stochastic fluctuations, this does not pose a theoretical problem. Moreover, in practice, such models are typically used in optimal execution problems arising in high-frequency trading, and are applied over time horizons short enough that negative prices are highly unlikely to occur.
\end{remark*}

The price's dynamics is therefore defined by

$$ dP_t = \frac{2k\displaystyle \frac{1}{N} \sum_{i=1}^N \alpha_t^i}{\displaystyle\left(X_0 - \int_0^t \frac{1}{N} \sum_{i=1}^N \alpha_s^i ds \right)^3} dt+ \sigma^0 dW_t^0. $$

We consider the total inventory of the $i$-th trader at time $t$ with $V_t^i$ given by the sum of the USDC and ETH holdings in USDC units
$$ V_t^i = Y_t^i + X_t^i P_t. $$
By Ito's Lemma, we have
\begin{equation}
\label{Ito_inventory}
  \begin{aligned}
dV_t^i &= dY_t^i + X_t^i dP_t + dX_t^i P_t \\
&= \left[ X_t^i \frac{2k\displaystyle \frac{1}{N} \sum_{j=1}^{N} \alpha_t^j}{\displaystyle\left(X_0 - \int_0^t \frac{1}{N} \sum_{j=1}^N \alpha_s^j ds \right)^3} \right] dt + X_t^i \sigma^0 dW_t^0 + P_t \sigma^i dW_t^i.
\end{aligned}
\end{equation}

We assume the agents are risk-neutral and seek to maximize their expected profit from operating in the decentralized market. That is, the $i$-th trader seeks to maximize
\begin{equation}\label{profit}
J^i(\alpha^1, \dots, \alpha^N) = \mathbb{E}\left[V_T^i - \int_0^T h(t, X_t^i) dt - l(X_T^i) \right],    
\end{equation}
where $h : [0,T] \times \mathbb{R} \to \mathbb{R}$, $x \to h(t,x)$ represents the cost of maintaining an inventory $x$ at time $t$ and $l : \mathbb{R} \to \mathbb{R}$ represents a terminal inventory cost.

\begin{obs*}
    Note that the mean field term in our problem appears in the form of the empirical distribution of the controls 
    \begin{equation}\label{control}
        \hat{q}_t^N = \frac{1}{N} \sum_{i=1}^{N} \alpha_t^i,
    \end{equation}
     rather than being represented through the empirical distribution of the states, denoted by $\mu$ in the literature. However, in what follows, we will state the more general problem considering both distributions. The results will still hold for our model as it is a particular case of the general problem, where the involved functions are constant in the variable associated with the probability measure over the state space.
\end{obs*} 

\subsection*{Working spaces}
In what follows, we will define the spaces and functions involved in the previous mean field game to have the necessary context to use the results of \citep{carmona2015weakmfgformulation}.

\begin{itemize}
    \item Let $(\Omega,\mathcal{F},\mathbb{F}=(\mathcal{F}_t)_{0\leq t\leq T},\mathbb{P})$ be a complete filtered probability space, where the filtration $\mathbb{F}$ supporting a 1-dimensional Wiener process $\bm{W}=(W_t)_{0\leq t\leq T}$ with respect to $\mathbb{F}$ and an initial condition $\xi\in L^2(\Omega,\mathcal{F}_0,\mathbb{P};\mathbb{R})$.
    \item Let $\mathcal{C} := C([0,T]; \mathbb{R})$ be the space of continuous functions with values in $\mathbb{R}$ departing from $[0,T]$, equipped with the supremum norm $\|x\| := \sup_{s \in [0,T]} |x(s)|$.
    \item Given $\mathcal{P}(\mathbb{R})$, the space of probability functions over $\mathbb{R}$, and a measurable function $\psi : \mathbb{R} \to \mathbb{R}$, define
    \begin{align*}
        \mathcal{P}_{\psi}(\mathbb{R}) &= \left\{ \mu \in \mathcal{P}(\mathbb{R}) : \int \psi d\mu < \infty \right\}, \\
        B_{\psi}(\mathbb{R}) &= \left\{ f:\Omega\to\mathbb{R} : \sup_\omega |f(\omega)|/\psi(\omega)<\infty \right\}.
    \end{align*}
    Define $\tau_{\psi}(\mathbb{R})$ as the weakest topology on $\mathcal{P}_{\psi}(\mathbb{R})$ that makes the map $\mu \to \int f d\mu$ continuous for each $f \in B_{\psi}(\mathbb{R})$.
    \item Let the control space $A \subset \mathbb{R}$ be a bounded subset and let $\mathbb{A} = \{\alpha : [0,T] \times \Omega \to A : \text{progressively measurable}\}$ be the set of admissible controls. 
    Throughout this work, all control functions $\alpha$ are considered in a closed-loop form, meaning they depend on the current state of the trader rather than just the initial condition or time. That is, each trader adopts a control strategy of the form $\alpha_t=\hat{\alpha}(t,X_t)$, where $\hat{\alpha}$ is a deterministic function mapping states to actions. This follows the formulation in \citep{carmona2015weakmfgformulation} and ensures that traders react optimally to market conditions at each moment.
    \item Finally, let $\mathcal{P}(A)$ be the space of probability functions over $A$ along with the weak topology $\tau(A)$ that makes the map $q \to \int_{A} f dq$ continuous for each $f \in B(A)$.
\end{itemize}

\begin{remark}
Since the stochastic processes in this model are defined on a filtered probability space $(\Omega, \mathcal{F}, \mathbb{P})$, all state and control functions implicitly depend on the event $\omega \in \Omega$. However, to simplify notation, we will omit this dependence throughout the paper.
\end{remark}

Thus, the functional \eqref{profit} is written as
$$ J^i(\alpha^1, \dots, \alpha^N) = \mathbb{E}\left[\int_0^T f(t, X_t^i, \mu, \hat{q}_t^N, \alpha_t^i) dt - l(X_T^i) \right], $$
where $\mu$ is a probability measure on the state space, $\hat{q}_t^N$ denotes the empirical distribution of $\alpha_t^1, \dots, \alpha_t^N$, and the function $f : [0,T] \times \mathbb{R} \times \mathcal{P}_{\psi}(\mathbb{R}) \times \mathcal{P}(A) \times A \to \mathbb{R}$ defined by
\begin{equation}\label{f}
  f(t, x, \mu, q, \alpha) =  \frac{2k x\displaystyle  \int_A \alpha dq}{\displaystyle \left(X_0 - \int_0^t \int_A \alpha dq \right)^3} - h(t, x),   
\end{equation}
represents the deterministic part of the dynamic of the trader's inventory \eqref{Ito_inventory} in the infinite case and $h$ represents the cost of maintaining, see \eqref{profit}.

Intuitively, if $N$ is large, due to the symmetry of the model, the contribution of player $i$ to the empirical distribution of controls $\hat{q}_t^N$ is negligible, and $\hat{q}_t^N$ can be treated as fixed. We consider the limit when the number of agents \( N \) tends to infinity, which simplifies the analysis by reducing the impact of individual fluctuations and focusing on the mean dynamics. Formally, we study the behavior of
\[
J(\alpha) := \mathbb{E}\left[\int_0^T f(t, X_t, \mu_t, q_t, \alpha_t) dt - l(X_T) \right],
\]
where $(\mu_t)_{0\leq t\leq T}$ is a flow of probability measures on the state space and $(q_t)_{0\leq t\leq T}$ is a flow of measures on the control space.


\subsection*{Mean field game}
Let $\xi\in L^2(\mathbb{R},\mathcal{F}_0,\mathbb{P};\mathbb{R})$ be an initial condition and we will note the distribution probability function, usually known as the law of $\xi$, as $\mathcal{L}(\xi)$. From a practical point of view, $\mu_0=\mathcal{L}(\xi)$ should be understood as the initial distribution of the population. We state the mean field game problem given by the following structure:
\begin{enumerate}
    \item Fix a flow $\bm{\mu}=(\mu_t)_{0\leq t\leq T}$ of probability measures on the state space and a flow $\bm{q}=(q_t)_{0\leq t\leq T}$ of measures on the control space;

    \item With $\mu$ and $q$ fixed, solve the standard optimal control problem:
    \begin{equation}
        \left\{ \begin{array}{lcc}\displaystyle 
             \sup_{\alpha} \mathbb{E}\,\left[\int_0^T f(t, X_t, \mu_t, q_t, \alpha_t) dt - g(X_T, \mu_T) \right] & \text{s.t.} \\
              dX_t = b(t, x, \mu, \alpha_t) dt + \sigma dW_t, \qquad X_0 = \xi, &
             \end{array}
   \right.
   \end{equation}
   where $f$ defined in \eqref{f},  $g : \mathbb{R} \times \mathcal{P}(\mathbb{R}) \to \mathbb{R}$ and $b : [0,T] \times \mathbb{R} \times \mathcal{P}_{\psi}(\mathbb{R}) \times A \to \mathbb{R}$;
   \item Find flows $\bm{\mu}$ and $\bm{q}$ such that $\mathcal{L}(\hat{X}^{\bm{\mu}, \bm{q}}_t)=\mu_t$ and $\mathcal{L}(\alpha_t)=q_t$ for all $t\in [0,T]$, if $\hat{X}^{\bm{\mu}, \bm{q}}_t$ is a solution of the above optimal control problem.
\end{enumerate}

This should be interpreted as the optimization problem faced by a single representative player in a game consisting of an infinite number of independently and identically distributed (i.i.d.) players. The second step provides the best response of a given player interacting with the statistical distribution of the states and of the controls of the other players if these statistical distributions are assumed to be given by $\mu$ and $q$, while the third step solves a specific fixed point problem in the spirit of the search for fixed points of the best response function.

Once the existence and, perhaps, the uniqueness of a fixed point are established, the second problem is to use this fixed point to construct approximate Nash equilibrium strategies for the original finite-player game. These strategies will be constructed from the optimal control for the problem in step 2., corresponding to the choice $(\mu,q)$ being the fixed point in step 1.

\subsubsection*{Solution to the mean field game}
For each continuous measure flow $\bm{\mu}=(\mu_t)_{0\leq t\leq T} \subset \mathcal{P}_{\psi}(\mathbb{R})$ and an admissible control $\alpha \in \mathbb{A}$, we define the probability measure $\mathbb{P}^{\bm{\mu}, \alpha}$ on $( \mathbb{R},\mathcal{F}_T )$ by
$$\frac{dP^{\bm{\mu}, \alpha}}{dP} = \mathcal{E}\left(\int_0 ^{\cdot} \sigma^{-1} b(s, X_s, \mu_s, \alpha_s) dW_s\right)_T,$$
where we use the notation $\mathcal{E}$ for the Doléans-Dade exponential of a martingale. Recall that if $\mathbf{M}=(M_t)_{0\leq t\leq T}$ is a local martingale, its Doléans-Dade exponential $\mathcal{E}(\mathbf{M})$ (or the stochastic exponential of $\mathbf{M}$) is defined by the formula:
$$\mathcal{E}(\mathbf{M})_t=\exp\left(M_t-M_0-\frac{1}{2}[M,M]_t\right),$$
where $([M,M]_t)_{0\leq t\leq T}$ stands for the quadratic variation of $\mathbf{M}$.
    

By Girsanov's theorem and the boundedness of $\sigma^{-1} b$, the process $W^{\bm{\mu}, \alpha}$ defined by
$$W_t^{\bm{\mu}, \alpha} := W_t - \int_0 ^t \sigma^{-1} b(s, X_s, \mu_s, \alpha_s) ds$$
is a Wiener process under $\mathbb{P}^{\bm{\mu}, \alpha}$, and 
$$dX_t = b(t, X_t, \mu_t, \alpha_t) dt + \sigma dW_t^{\bm{\mu}, \alpha}.$$
That is, under $\mathbb{P}^{\bm{\mu}, \alpha}$, $X$ is a weak solution to the state equation. Note that $\mathbb{P}^{\bm{\mu}, \alpha}$ and $\mathbb{P}$ are equal on $\mathcal{F}_0$; in particular, the law of $X_0 = \xi$ remains $\lambda_0$. Moreover, $\xi$ and $W$ remain independent under $\mathbb{P}^{\bm{\mu}, \alpha}$.

We will note by $\mathbb{P}^{\bm{\mu}, \alpha} \circ X^{-1}=X_{*}(\mathbb{P}^{\bm{\mu}, \alpha})$ and by $\mathbb{P}^{\bm{\mu}, \alpha} \circ \alpha^{-1}=(\alpha)_{*}(\mathbb{P}^{\bm{\mu}, \alpha})$ the push-forward measures of $\mathbb{P}^{\bm{\mu}, \alpha}$ by $X$ and of $\mathbb{P}^{\bm{\mu}, \alpha}$ by $\alpha$, respectively (see for ex. \citep{bogachev2005measuretheory}). Then, $\mathbb{P}^{\bm{\mu}, \alpha} \circ X^{-1} = \mu$ is asking that the events induced by $X$ be the same as those identified by $\mu$; and $\mathbb{P}^{\bm{\mu}, \alpha} \circ \alpha^{-1} = q$ is asking that the events induced by $\alpha$ be the same as those identified by $q$.


Given flows $\bm{\mu}=(\mu_t)_{0\leq t\leq T}$ and $\bm{q}=(q_t)_{0\leq t\leq T}$ of probability measures on $\mathbb{R}$ and on $A$ resp., and given a control $\alpha \in \mathbb{A}$ we define the associated expected reward by
$$ J^{\bm{\mu}, \bm{q}}(\alpha) := \mathbb{E}^{\bm{\mu}, \alpha}\left[\int_0^T f(t, X_t, \mu_t, q_t, \alpha_t) dt + g(X_T, \mu_T) \right], $$
where $\mathbb{E}^{\bm{\mu}, \alpha}$ denotes the expectation with respect to the measure $\mathbb{P}^{\bm{\mu}, \alpha}$. Given $\bm{\mu}$ and $\bm{q}$ as before, we face a standard stochastic optimal control problem, whose value is given by
$$ V^{\bm{\mu}, \bm{q}} = \sup_{\alpha \in \mathbb{A}} J^{\bm{\mu}, \bm{q}}(\alpha). $$

Formally, we define a solution to the MFG as follows:
\begin{Def}
    We say that a pair of flows $(\bm{\mu},\bm{q})$ form a \textit{solution of the MFG} if there exists $\alpha \in \mathbb{A}$ such that $V^{\bm{\mu}, \bm{q}} = J^{\bm{\mu}, \bm{q}}(\alpha)$, $\mathbb{P}^{\bm{\mu}, \alpha} \circ X^{-1} = \mu$, and $\mathbb{P}^{\bm{\mu}, \alpha} \circ \alpha^{-1} = q$ for almost every $t$.


\end{Def}

\section{Mean Field Game solution}
We will now prove the existence of a MFG solution for the token price model in a liquidity pool with infinite traders and, additionally, find approximate Nash equilibria when the pool operates with a finite number of traders.

We assume that the initial ETH inventories \( X_0^i \) are i.i.d and have a common distribution \( \lambda_0 \in \mathcal{P}(\mathbb{R}) \) satisfying 
$$
\int_{\mathbb{R}} e^{p|x|} \lambda_0 (dx) < +\infty \qquad \forall p > 0. 
$$
Suppose \( A \subset \mathbb{R} \) is a compact interval containing the origin, \( \sigma > 0 \), the costs $h$ and $l$ defined in \eqref{profit} are measurable, and  there exists \( c_1 > 0 \) such that
\begin{equation}
    \label{Sup: Acotacion}
    |h(t,x)| + |l(x)| \leq c_1 e^{c_1 |x|} \hspace{0.3cm} \text{ for all } (t,x) \in [0,T] \times \mathbb{R}.
\end{equation}
With the notation of the previous section,  \( (t,x,\mu,q,a) \in [0,T] \times \mathbb{R} \times \mathcal{P}_{\psi}(\mathbb{R}) \times \mathcal{P}(A) \times A \), we define the following functions:
\begin{align*}
    b(t,x,\mu,a) &:= a, \qquad 
    g(\mu, x) := l(x) ,\qquad
    \psi(x) := e^{c_1 |x|},\\
    f(t,x,\mu,q,a) &:= x \frac{2k \displaystyle \int_A r dq_t (r)}{(X_0 - \displaystyle \int_0^t \int_A rdq_s(r) ds)^3} - h(t,x) .
\end{align*}
Let us see that the hypotheses of Theorem 3.5 in \citep{carmona2015weakmfgformulation} are satisfied, which guarantees the existence of a MFG solution.

\begin{cor}
\label{sup: S}
The following conditions are satisfied:
\begin{enumerate}
    \item[\textbf{\hyperref[sup: S.1]{(S.1)}}] 
    \label{sup: S.1}
    The control space \( A \) is a compact and convex subset of \( \mathbb{R} \) and the drift \( b : [0,T] \times \mathbb{R} \times \mathcal{P}_{\psi}(\mathbb{R}) \times A \to \mathbb{R} \) is continuous.
    \item[\textbf{\hyperref[sup: S.2]{(S.2)}}] 
    \label{sup: S.2}
    There exists a strong solution \( X \) to the driftless state equation
    \begin{equation}
    \label{dynamic of X without drift}
    dX_t = \sigma dW_t, \hspace{0.5cm} X_0 = \xi,
    \end{equation}
    such that \( \mathbb{E}[\psi^2 (X)] < +\infty \), \( \sigma > 0 \), and \( \sigma^{-1} b(a) \) is uniformly bounded.
    \item[\textbf{\hyperref[sup: S.3]{(S.3)}}] 
    \label{sup: S.3}
    The cost/benefit function \( f : [0,T] \times \mathbb{R} \times \mathcal{P}_{\psi}(\mathbb{R}) \times \mathcal{P}(A) \times A \to \mathbb{R} \) is such that \( (t,x) \to f(t,x,\mu,q,a) \) is progressively measurable for each \( (\mu,q,a) \) and \( a \to f (t,x,\mu,q,a) \) is continuous for each \( (t,x,\mu,q) \). The terminal cost/reward function \( g : \mathbb{R} \times \mathcal{P}_{\psi}(\mathbb{R}) \to \mathbb{R} \) is Borel measurable for each \( \mu \).
    \item[\textbf{\hyperref[sup: S.4]{(S.4)}}] 
    \label{sup: S.4}
    There exists \( c > 0 \) such that
    $$ |g(x,\mu)| + |f(t,x,\mu,q,a)| \leq c \psi(x), \quad \forall (t,x,\mu,q,a) \in [0,T] \times \mathbb{R} \times \mathcal{P}_{\psi}(\mathbb{R}) \times \mathcal{P}(A)\times A. $$
    \item[\textbf{\hyperref[sup: S.5]{(S.5)}}] 
    \label{sup: S.5}
    The function \( f \) is of the form 
    $$ f(t,x,\mu,q,a) = f_1(t,x,\mu,a) + f_2(t,x,\mu,q). $$
\end{enumerate}
\end{cor}

\begin{proof}
    Let us verify each of the above statements one by one.
    \begin{enumerate}
        \item[\textbf{\hyperref[sup: S.1]{(S.1)}}] 
        The quantities of tokens traded by the agents cannot be infinite since \( A \) is bounded, i.e., there exists \( M \) such that \( |\alpha(t,\omega)| \leq M\),  \(\forall (t,\omega) \in [0,T] \times \Omega \). The control space \( A \) is a closed interval containing the origin, therefore it is compact and convex. The drift is a continuous function, as it is the identity.
        \item[\textbf{\hyperref[sup: S.2]{(S.2)}}] 
        The function \( \sigma^{-1}b(a) = \sigma^{-1} a \) is uniformly bounded since \( A \) is compact. Moreover, since \( \psi_0(x) = e^{p|x|} \), \( \lambda_0 \) is such that \( \int_{\mathbb{R}^d} \psi_0(x)^2 \lambda_0(dx) < \infty \) and \( \sigma \) is constant, the hypotheses of Lemma 5.2 of \citep{carmona2015weakmfgformulation} are verified.
    \end{enumerate}

    For \textbf{\hyperref[sup: S.3]{(S.3)}} and \textbf{\hyperref[sup: S.4]{(S.4)}}, we need to ensure the model is well-posed. Given that the set of admissible controls $A$ is compact we can take the bound used in \textbf{\hyperref[sup: S.1]{(S.1)}} as $M=\displaystyle\max_{a \in A} |a|$. Then, we have
    \begin{align*}
            X_t = X_0 - \int_0^t \int_A rdq_s(r) ds &\geq X_0 - TM, \quad  0\leq t\leq T.
    \end{align*}
    Therefore, we restrict the controls such that 
    \begin{equation}\label{condA}
        M < 
        \frac{X_0}{T}
    \end{equation}
    Then we have that there exist $\epsilon_0>0$ for which
    $$X_t > \epsilon_0, \quad \hbox{ for } 0\leq t\leq T.$$
    Using the above, we have
        \begin{align*}
           \left| \frac{2k \displaystyle\int_A r dq_t(r)}{\displaystyle\Big(X_0 - \int_0^t \int_A rdq_s(r) ds\Big)^3} \right| &\leq \frac{2k M}{\epsilon_0^3}.
       \end{align*}

    \begin{enumerate}
        \item[\textbf{\hyperref[sup: S.3]{(S.3)}}]
        The function \eqref{f} is progressively measurable as a function of \( (t,x) \) since it is well-defined and continuous, and is continuous as a function of \( a \) since it is constant. Furthermore, the terminal cost \( g(x, \mu) \) is continuous and hence Borel measurable.
        \item[\textbf{\hyperref[sup: S.4]{(S.4)}}] By equation \eqref{Sup: Acotacion} and the previous calculation, we have
        \begin{align*}
            |g(x, \mu) + f(t,x,\mu,q,a)| \leq c_1 e^{c_1 |x|} + |x| \frac{2k M}{\epsilon_0^3} \leq \tilde{c} e^{\tilde{c} |x|}
        \end{align*}
        for \( \tilde{c} = \max \left\{ c_1, \frac{2k M}{\epsilon_0^3} \right\} \). We have what we wanted for \( \psi(x) = e^{\tilde{c} |x|} \).
        \item[\textbf{\hyperref[sup: S.5]{(S.5)}}]
        Note that \( f(t,x,\mu,q,a) = f_1(t,x,\mu,a) + f_2(t,x,\mu,q) \)
        where \( f_1(t,x,\mu,a) = 0 \) and 
        $$f_2(t,x,\mu,q) = x \frac{\displaystyle 2k \int_A r dq_t(r)}{\displaystyle \Big(X_0 - \int_0^t \int_A rdq_s(r) ds\Big)^3} - h(t,x).$$
    \end{enumerate}
\end{proof}

Some additional conditions are needed for the existence results. The Hamiltonian \( h : [0,T] \times \mathbb{R} \times \mathcal{P}_{\psi}(\mathbb{R}) \times \mathcal{P}(A) \times \mathbb{R}\times A \to \mathbb{R} \), the maximized Hamiltonian \( H : [0,T] \times \mathbb{R} \times \mathcal{P}_{\psi}(\mathbb{R}) \times \mathcal{P}(A) \times \mathbb{R} \to \mathbb{R} \), and the set where the supremum is attained are defined by
\begin{align*}
    h(t,x,\mu,q,z,a) &:= f(t,x,\mu,q,a) + z \sigma^{-1}b(t,x,\mu,a), \\
    H(t,x,\mu,q,z) &:= \sup_{a \in A} h(t,x,\mu,q,z,a), \\
    A(t,x,\mu,q,z) &:= \{ a \in A : h(t,x,\mu,q,z,a) = H(t,x,\mu,q,z) \},
\end{align*}
respectively. Note that by assumption \textbf{(S.5)}, \( A(t,x,\mu,q,z) \) does not depend on \( q \), so we will use the notation \( A(t,x,\mu,z) \). Furthermore, by assumptions \textbf{(S.1)} and \textbf{(S.3)}, the compactness of \( A \) and the continuity of \( h \) in the variable \( a \), we have that \( A(t,x,\mu,z) \) is non-empty.


\begin{cor}
The following condition holds:
\begin{description}
  \item[\textbf{\hyperref[sup: C]{(C)}}]
  \label{sup: C}
 For each \( (t,x,\mu,z) \), the set \( A(t,x,\mu,z) \) is convex.
\end{description}
\end{cor}
\begin{proof}
    It holds whenever \( b \) is affine as a function of \( a \) and \( f \) is also concave as a function of \( a \). The first condition follows from \( b(a) = a \); the second because it is constant.

    Let \( a,b \in A(t,x,\mu,q,z) \) and \( \lambda \in [0,1] \). We need to prove that \( c := (1-\lambda) a + \lambda b \in A(t,x,\mu,q,z) \).
 Since \( a,b \in A(t,x,\mu,q,z) \), we have that  
    \begin{align*}
        h(t,x,\mu,q,z,a) &= H(t,x,\mu,q,z) = h(t,x,\mu,q,z,b).
    \end{align*}
    Then,  
    \begin{align*}
        h(t,x,\mu,q,z,c)
        &= f(t,x,\mu,q,((1-\lambda) a + \lambda b)) + z \sigma^{-1}((1-\lambda) a + \lambda b)\\
        &\geq (1-\lambda) f(t,x,\mu,q,a) + \lambda f(t,x,\mu,q,b) + z \sigma^{-1}((1-\lambda) a + \lambda b)\\
        &= (1-\lambda) (f(t,x,\mu,q,a) + z \sigma^{-1} a) + \lambda (f(t,x,\mu,q,b) + z \sigma^{-1} b)\\
        &= (1-\lambda) h(t,x,\mu,q,z,a) + \lambda h(t,x,\mu,q,z,b)\\
        &= (1-\lambda) h(t,x,\mu,q,z,a) + \lambda h(t,x,\mu,q,z,a)\\
        &= h(t,x,\mu,q,z,a) = H(t,x,\mu,q,z).
    \end{align*}
    Since \( a \) was already a maximizer, we must have the equality:  
    \[
    h(t,x,\mu,q,z,c) = h(t,x,\mu,q,z,a) = H(t,x,\mu,q,z).
    \]
\end{proof}

It will be useful to have a notation for the law without drift and the set of equivalent laws,
\begin{align*}
    \mathcal{X} &:= \mathbb{P} \circ X^{-1} \in \mathcal{P}_{\psi}(\mathbb{R}), \\
    \mathcal{P}_\mathcal{X} &:= \{ \mu \in \mathcal{P}_{\psi}(\mathbb{R}) : \mu \sim \mathcal{X} \}.
\end{align*}

\begin{cor}
The following holds:
\begin{description}
  \item[\textbf{\hyperref[sup: E]{(E)}}]
  \label{sup: E}
    For each \( (t,x) \in [0,T] \times \mathbb{R} \), the following maps are sequentially continuous, using \( \tau_{\psi}(\mathbb{R}) \) on \( \mathcal{P}_\mathcal{X} \) and the weak topology on \( \mathcal{P}(A) \):
\begin{align*}
    \mathcal{P}_\mathcal{X} \times A &\ni (\mu, a) \mapsto b(t,x,\mu,a), \\
    \mathcal{P}_\mathcal{X} \times \mathcal{P}(A) \times A &\ni (\mu,q,a) \mapsto f(t,x,\mu,q,a), \\
    \mathcal{P}_\mathcal{X} &\ni \mu \to g(x,\mu).
\end{align*}
\end{description}
\end{cor}
\begin{proof} 
    For each \( (t,x) \in [0,T] \times \mathbb{R} \), we have that \( a \mapsto b(\mu, a) = a \) is sequentially continuous since it is continuous. Additionally, given that \( q \mapsto \int_A r dq(r) \) is continuous, the map \( (\mu,q,a) \mapsto f(t,x,\mu,q,a) \) is also continuous. Finally, \( g \) is constant with respect to \( \mu \).
\end{proof}
\begin{prop}
Let's recall our price impact model in a liquidity pool with infinite traders with state dynamic
$$dX_t = \alpha_t dt + \sigma dW_t,  $$
and associated expected reward
$$ J^{\mu, q}(\alpha) = \mathbb{E}^{\mu, \alpha}\left[\int_0^T f(t, X_t, \mu_t, q_t, \alpha_t) dt + g(X_T, \mu_T) \right], $$
for $f,g$ given at the beginning of the section. There exists a solution to the MFG induced by this model.
\end{prop}
\begin{proof}
    By Theorem 3.5 of \citep{carmona2015weakmfgformulation}, since conditions \textbf{\hyperref[sup: C]{(C)}} and \textbf{\hyperref[sup: E]{(E)}} are satisfied, there exists a solution to the MFG.
\end{proof}

The failure to satisfy the necessary assumptions for uniqueness results (Theorem 3.8, \citep{carmona2015weakmfgformulation}) does not imply its inexistence; rather, it indicates that our model does not fit these specific hypotheses. It would be necessary to adjust the existing model or explore new results to address this discrepancy.

\section{Approximation to Nash}

In this section, we shift our focus to the finite-player version of the game and study the convergence of the system towards a Nash equilibrium as the number of players increases. Specifically, we aim to show that the solution obtained from the mean field game provides a foundation for constructing approximate Nash equilibria for the original $N$-player game. 

Let's recall that the empirical measure map \( e_n : \Omega^n \to \mathcal{P}(\Omega) \) is given by
$$ e_n(\omega_1,...,\omega_n) = \frac{1}{n} \sum_{j=1}^n \delta_{\omega_j}. $$

Given the measurable spaces \( E \) and \( F \), we say that \( f : \mathcal{P}(\Omega) \times E \to F \) is \textit{empirically measurable} if
$$ (x,y) \in \Omega^n \times E \to f(e_n(x),y) \in F $$
is jointly measurable for all \( n \geq 1 \).

\begin{cor}
The following hypotheses are satisfied:
\begin{enumerate}
    \item[\textbf{\hyperref[sup: F.1]{(F.1)}}]
    \label{sup: F.1}
    The drift \( b = b(t,x,a) \) has no mean field term.
    \item[\textbf{\hyperref[sup: F.2]{(F.2)}}]
    \label{sup: F.2}
    The functions \( b \), \( f \), and \( g \) are empirically measurable using the progressive \(\sigma\)-field on \( [0,T] \times \mathbb{R} \) and the Borel \(\sigma\)-fields elsewhere.
    \item[\textbf{\hyperref[sup: F.3]{(F.3)}}]
    \label{sup: F.3}
    For each \( (t,x) \), the following functions are continuous at each point satisfying \( \mu \sim \mathcal{X} \):
    $$ \mathcal{P}_{\psi}(\mathbb{R}) \times \mathcal{P}(A) \times A \ni (\mu,q,a) \mapsto f(t,x,\mu,q,a), $$
    $$ \mathcal{P}_{\psi}(\mathbb{R}) \ni \mu \to g(x,\mu). $$
    \item[\textbf{\hyperref[sup: F.4]{(F.4)}}]
    \label{sup: F.4}
    There exists \( c > 0 \) such that, for all \( (t,x,q,a) \),
    $$ |g(x,\mu)| + |f(t,x,\mu,q,a)| \leq c \left( \psi(x) + \int \psi d\mu \right). $$
\end{enumerate}
\end{cor}

\begin{proof}
    Note that \textbf{\hyperref[sup: F.1]{(F.1)}} is clear. Let \( E = [0,T] \times \mathbb{R} \times \mathcal{P}(A) \times A \) and let \( (\vec{x},t,x,q,a) \in \mathbb{R}^n \times E \). We want to see that
    \begin{align*}
        (\vec{x},t,x,a) &\mapsto b(t,x,e_n(\vec{x}),a), \\
        (\vec{x},t,x,q,a) &\mapsto f(t,x,e_n(\vec{x}),q,a), \\
        (\vec{x},x) &\mapsto g(x, e_n(\vec{x}))
    \end{align*}
    are measurable, where
    \( \vec{x} \in \mathbb{R}^n \mapsto e_n(\vec{x}) := \displaystyle\frac{1}{n} \sum_{i=1}^n \delta_{x_i} \in \mathcal{P}_{\psi}(\mathbb{R}) \).

    Since \( b,f \) and \( g \) in our case do not depend on the state law coordinate and are continuous, particularly measurable, and we are considering the progressive \(\sigma\)-fields on \( [0,T] \times \mathbb{R} \) and the Borel \(\sigma\)-fields on the remaining coordinates, the composition with \( e_n \) results in a measurable function. 
    Therefore, \textbf{\hyperref[sup: F.2]{(F.2)}} is satisfied.
    
    Furthermore, since \( q \to \int_A r dq(r) \) is continuous, \textbf{\hyperref[sup: F.3]{(F.3)}} follows.
    
    Finally, using that \( \int \psi d\mu \) is positive if we consider \( \psi(x) = e^{c|x|} \) and performing the same computation as in \textbf{\hyperref[sup: S.4]{(S.4)}}, we see that \textbf{\hyperref[sup: F.4]{(F.4)}} is also satisfied.
\end{proof}

\begin{prop}
There exists an approximate Nash equilibrium for the game with finite players in the sense that there exists a sequence \( \varepsilon_n \geq 0 \)
with \( \varepsilon_n \to 0 \) such that, for \( 1 \leq i \leq n \) and \( \beta \in \mathbb{A}_n \),
$$ J_{n,i}(\alpha^1,...,\alpha^{i-1},\beta,\alpha^{i+1},...,\alpha^n) \leq J_{n,i}(\alpha^1,...,\alpha^n) + \varepsilon_n. $$
\end{prop}

\begin{proof}
    By Theorem 5.1 of \citep{carmona2015weakmfgformulation}, given \( (\hat{\mu},\hat{q}) \) a solution to the MFG with corresponding closed-loop control $\alpha_t=\hat{\alpha}(t,x)$ (see Remark 3.2 of \citep{carmona2015weakmfgformulation}) that exists by the previous result, the strategies \( \alpha_t^i := \hat{\alpha}(t,X^i_t) \) form an approximate Nash equilibrium for the finite player game.
\end{proof}

\section{Numerical Analysis and Practical Implications}
\label{sec:numerics}
This section presents a sequence of numerical experiments illustrating the behavior of the Mean Field Game approximation for the AMM liquidation model.

We perform a detailed numerical analysis to explore the behavior of the optimal trading policy in our mean field game framework with a constant product AMM. Specifically, we are interested in whether the traders' optimal control $\alpha$ can effectively move the market in the presence of the strong nonlinear price impact.

The instantaneous reward function is given by:
\[
f(t,x,\mu,q,\alpha) = \frac{2kx \int_A \alpha\,dq}{\left(X_0 - \int_0^t \int_A \alpha\,dq\right)^3} - h(t,x),
\]
where the cubic denominator reflects the AMM's liquidity pool depletion mechanism. This denominator dominates the impact of the control $\alpha$, potentially flattening the reward landscape.


Our objectives are threefold: 
\begin{enumerate}
    \item to show how the finite--player system approximates the MFG equilibrium; 
    \item to explore the stability of the equilibrium under small perturbations; and 
    \item to evaluate how different modeling choices (action set, cost structure, and market depth) affect the qualitative shape of the equilibrium control.
\end{enumerate}

\begin{obs*}
To isolate the fixed--point structure and validate the numerical MFG pipeline, we run throughout this section a tractable aggregate--impact proxy in which the population effect enters through the empirical mean control. This specification preserves the mean field coupling while avoiding additional nonlinear feedback effects induced by the AMM reserve dynamics. Accordingly, the experiments below should be read as algorithmic and game--theoretic sanity checks (stability, $\varepsilon$--Nash, and finite--$N$ convergence) rather than a calibrated quantitative study of AMM price impact. A fully endogenous AMM--feedback (``realistic impact'') implementation is left for future work.
\end{obs*}

All experiments use the fixed--point solver described in 
Section~\ref{sec:algorithm}, and all simulations are performed in the
''analytic'' mode (aggregate--impact).

\subsection*{Experiment 0: the frictionless model is degenerate}

We begin with the frictionless version of the model, where all costs are 
switched off:
\[
\lambda = \eta_{\mathrm{term}} = \rho = 0,
\qquad
A = [-3,3].
\]

When $\lambda=\eta=\rho=0$, the one--step value difference becomes nearly flat in the control variable, so the best--response map provides no robust selection mechanism.
As a result, the fixed point computation is \emph{ill-conditioned}: the solver converges to an arbitrary admissible profile (typically close to the midpoint of $A$ in symmetric settings), driven mainly by discretization and numerical asymmetries rather than by economic forces.
This behavior is consistent with the fact that, in the absence of inventory penalties, terminal incentives, or control costs, the control problem admits a \emph{continuum of near-indifferent solutions}, hence equilibrium selection is not identifiable without regularization.

This experiment motivates the introduction of frictions, which regularize the model and 
produce an interior, economically meaningful equilibrium.  
In the following, we therefore focus on a parameter regime with nonzero inventory 
and control costs.

The numerical outcomes are summarized in
Table~\ref{tab:exp0}, which reports equilibrium magnitude, stability indicators, and finite--player deviations across different admissible action sets.

\subsection*{Baseline scenario and stability checks}
\label{sec:baseline}

We now consider the ''regularized'' model that serves as our baseline example.
The parameters are:
\[
X_0 = 10,\quad 
Y_0 = 10\,000,\quad 
N = 50,\quad
dt = 0.02,\quad T = 1,
\]
\[
\lambda = 2,\qquad
\eta_{\mathrm{term}} = 0.3,\qquad
\rho = 5,\qquad
A = [-1.5,1.5] \text{ (61 points)},\qquad
\sigma_i = 0.02.
\]
In this regime the optimal one--step control 
\(
a^\star(x) \approx -\frac{\eta_{\mathrm{term}}}{\rho}(x - x^\star)
\)
lies inside the admissible action set.  
This eliminates the bang--bang behavior observed earlier and results in a 
smooth, interior equilibrium $q^\star$.

We report four diagnostics.

\paragraph*{1. Baseline summary.}
Figure~\ref{fig:qstar_precio_baseline} illustrates the equilibrium structure in the baseline scenario. The aggregate equilibrium control satisfies $q^*(t) < 0$ for all $t \in [0, T]$, with values in the range $[-0.61, -0.57]$; this reflects that, under the calibrated inventory costs and terminal penalty, the optimal coordinated strategy of the agents consists of selling ETH steadily throughout the entire horizon. The price trajectory $P(t)$ induced by $q^*$ is strictly decreasing, from $P_0 \approx 998$ to $P_T \approx 892$, showing that the net selling flow is transmitted monotonically to the AMM through the constant product mechanism. The uniform negativity of $q^*(t)$ and the strict decrease of $P(t)$ are intrinsically linked: the MFG equilibrium reflects a coordinated liquidation pressure that the AMM absorbs through a monotone downward price adjustment, without pool depletion ($L(t) \geq 10.01$ for all $t$).

Table~\ref{tab:baseline} summarizes the equilibrium magnitude, stability indicators, and finite-player deviations across the three action set configurations.
All unilateral improvement values satisfy $\max_i \Delta J_i \approx 0$, confirming the $\varepsilon$--Nash character of the baseline equilibrium across all tested action sets.

\paragraph*{2. Stability to perturbations.}
For small perturbations $q^\star + \varepsilon \delta q$ we compute the 
finite--difference sensitivity
\[
\gamma(\varepsilon)
=
\frac{\|\,\mathcal{F}(q^\star+\varepsilon \delta q) - q^\star\,\|}
{\varepsilon},
\]
where $\mathcal{F}$ denotes one iteration of the fixed--point map.
Fig.~\ref{fig:perturb} shows $\gamma(\varepsilon)$ decreasing steadily as 
$\varepsilon\downarrow 0$, indicating that $q^\star$ is a stable fixed point. We evaluate $\gamma$ at six perturbation magnitudes $\epsilon\in\{10^{-3}, 3*10^{-3}, 5*10^{-3}, 10^{-2}, 3*10^{-2}, 5*10^{-2}\}$ and report the mean $\overline{\gamma}$ in the summary tables.

\paragraph*{3. Unilateral deviations ($\epsilon$--Nash).}
Using the ''frozen policy'' mechanism developed in Section~\ref{sec:algorithm}, 
we recompute the objective of each trader when only that trader is allowed  to deviate optimally while all others remain at their baseline trajectories. The resulting histogram of improvements $\Delta J_i$ is shown in Fig.~\ref{fig:unilateral}.

All values cluster tightly at $0$, demonstrating that $q^\star$ is an 
$\varepsilon$--Nash equilibrium with $\varepsilon$ numerically indistinguishable from zero. 

\paragraph*{4. Finite--player approximation.}
To assess whether the mean--field equilibrium $q^\star$ provides a good approximation for large but finite populations, we quantify the discrepancy between the empirical average control 
\(
\bar\alpha^N(t) = \frac1N \sum_{i=1}^N \alpha_i(t)
\)
and $q^\star$. For a given number of agents $N\in\{25, 50, 100, 150, 200\}$, we define the finite--$N$ gap as
\[
\mathrm{gap}(N) := \|\bar{\alpha}^N - q^\star\|_2.
\]
Reported uncertainty bands correspond to the standard error of the mean (SEM), computed over multiple random seeds as $\mathrm{SEM} = \sigma / \sqrt{S}$, where $S$ denotes the number of seeds.

Fig.~\ref{fig:finiten} shows that $\|\bar\alpha^N - q^\star\|_2$ decays approximately as a power law in $N$, indicating standard propagation--of--chaos scaling.

\subsection*{Sensitivity experiments}
Quantitative results for all sensitivity experiments are reported in
Tables~\ref{tab:expA}, \ref{tab:expB}, \ref{tab:expC} and~\ref{tab:exp_D}, respectively.

Unless otherwise stated, all metrics are computed at the mean field equilibrium $q^\star$
and averaged over independent runs.
\subsubsection*{Experiment A: effect of the action set}

We now vary the size of the action set:
\[
A_{\text{small}} = [-1,1],\qquad
A_{\text{medium}} = [-1.5,1.5]\ (\text{baseline}),\qquad
A_{\text{large}} = [-10,10].
\]
The results are intuitive:
\begin{itemize}
    \item With $A_{\text{large}}$, the equilibrium flow is smooth and close to linear.
    \item With $A_{\text{medium}}$, we recover the baseline interior solution.
    \item With $A_{\text{small}}$, the system approaches the bang--bang behavior of the frictionless case unless the control cost $\rho$ is sufficiently large.
\end{itemize}

This confirms that the admissible action set plays a nontrivial role when costs 
are weak, but becomes largely irrelevant once $\rho$ is in the baseline range.


\paragraph{Bang--bang saturation.}
When the unconstrained myopic control $a^\star(x)$ falls outside the admissible action set $A=[A_{\min},A_{\max}]$, the discrete argmax policy saturates at the boundary, yielding a bang--bang regime.
To quantify this effect we report the saturation rate
\begin{equation}
\mathrm{sat\_rate}
\;:=\;
\frac{1}{NT}\sum_{i=1}^N\sum_{t=1}^T
\mathbf{1}\!\left\{\alpha_{i,t}\in\{A_{\min},A_{\max}\}\right\},
\label{eq:sat_rate}
\end{equation}
i.e., the fraction of agent--time controls that hit the boundary of $A$.
We observe higher $\mathrm{sat\_rate}$ precisely in configurations where $A$ is tight and/or the cost parameters push $a^\star(x)$ outside $A$ (cf. Experiments~A and~C).



\subsubsection*{Experiment B: market depth (pool size)}

We compare the baseline pool size $X_0=10$ with a deeper pool $X_0 = 100$.  
Because individual trades represent a smaller fraction of the reserves,  
the price impact of each action is attenuated, leading to a 
smaller--magnitude equilibria $q^\star$ and reduced liquidity depletion.
The finite--$N$ gap also decreases faster for the deeper pool, reflecting 
the lower variance of price shocks relative to market depth.



\subsubsection*{Experiment C: effect of cost parameters}

Finally, we vary the three cost components independently.

\paragraph*{Increasing $\lambda$ (running inventory cost).}
Makes agents more reluctant to hold inventory intra--period, producing a 
more negative $q^\star$ and a steeper price path.

\paragraph*{Increasing $\eta_{\mathrm{term}}$ (terminal inventory penalty).}
Shifts the control primarily near the terminal time, resulting in a 
more front--loaded liquidation pattern.

\paragraph*{Increasing $\rho$ (control cost).}
Regularizes the problem most strongly: large $\rho$ suppresses bang--bang 
behavior and produces an interior optimum even in wide action sets.

Overall, these three knobs shape the magnitude, boundary saturation, and stability properties of the equilibrium in predictable ways, providing an economically interpretable map of model behavior.

\subsubsection*{Experiment D: Changes in the terminal holdings target}
Finally, we modify the terminal target $x^*$.

\paragraph{Reversal of the equilibrium direction.}
By changing the traders' terminal target from $x^* = 0$ (liquidation) to $x^* = 15$ (ETH accumulation), the mean field equilibrium exhibits a mirror behavior of the baseline scenario (Figure~\ref{fig:qstar_mirror_xstar}): the aggregate control $q^*(t)$ becomes strictly positive ($q^*(t) \approx 0.28$ for all $t$), and the price trajectory $P(t)$ is strictly increasing. This result confirms that the equilibrium direction---coordinated selling or buying---is not a model artifact but a direct consequence of the incentive structure: when the terminal cost penalizes inventories below $x^*$, agents buy ETH to reach their target, generating buying pressure that the AMM absorbs through a monotone upward price adjustment.

\subsection*{Summary of numerical findings}

The numerical experiments reported in this section lead to the following observations:
\begin{enumerate}
    \item In the absence of regularization, the model admits multiple extreme equilibria, resulting in an underdetermined and economically degenerate regime. Introducing inventory and control costs selects a unique and smooth mean field equilibrium.
    
    \item The computed equilibrium $q^\star$ is locally stable: small perturbations of the aggregate flow generate bounded responses, and unilateral deviations yield negligible individual gains, providing numerical evidence of an $\varepsilon$--Nash equilibrium.
    
    \item Finite--player simulations converge to the mean field solution as the population size increases. The finite--$N$ gap decreases systematically and remains within narrow confidence bands across independent random seeds.
    
    \item Equilibrium trading intensity and aggregate inventory levels respond monotonically to variations in cost parameters and market depth, as reflected by equilibrium norms and liquidity metrics.
    
    \item Sensitivity experiments show that regularization parameters control the qualitative regime of the equilibrium (interior vs.\ boundary saturation), while other parameters primarily rescale equilibrium magnitudes without affecting stability.

    \item The terminal target $x^*$ controls the direction of the equilibrium: changing $x^*$ from liquidation ($x^*=0$) to accumulation ($x^*>x_0$) reverses the sign of $q^*$ and the monotonicity of $P(t)$, confirming that the equilibrium direction is driven by the incentive structure rather than by model artifacts.
\end{enumerate}

\section{Conclusion}

In this paper, we developed a novel application of mean field game (MFG) theory to model price formation in decentralized liquidity pools governed by constant product market makers. Our approach adapts traditional price impact models from centralized order book markets to a decentralized context, capturing the strategic interactions of traders and the unique price discovery mechanisms of automated market makers (AMMs). We established the existence of solutions to the MFG and of approximate Nash equilibria, providing a theoretical foundation for analyzing trader behavior and market dynamics in decentralized finance (DeFi) environments.

The approximate Nash equilibrium results extend classical equilibrium concepts, offering a framework capable of handling realistic settings where exact solutions may be difficult to achieve. These results validate the model's consistency and motivate further research exploring mean field techniques in DeFi.

By linking trader actions to changes in the pool's reserves, the model reflects how individual trades influence the execution price through the AMM's invariant curve. This feature — absent in classical price impact models based on centralized order books — provides a more accurate representation of price dynamics in decentralized markets, and clarifies the incentive structure that drives participation in automated market makers.

In addition to the theoretical analysis, we conducted numerical experiments to test how the system reacts to changes in the reward scaling factor. The simulations show that, once minimal inventory or control costs are introduced, the mean field equilibrium $q^\star$ is uniquely selected and numerically stable. Small perturbations of the aggregate order flow induce controlled responses, and unilateral deviations yield negligible individual gains, providing consistent numerical evidence of an $\varepsilon$--Nash equilibrium.

The numerical experiments also highlight an important structural feature of AMM--based models. In the absence of regularization, the system becomes underdetermined and collapses into extreme trading regimes, reflecting a genuine degeneracy of the frictionless formulation. Regularization therefore plays a fundamental modeling role rather than serving as a numerical convenience. When control constraints are active, the equilibrium may enter a bang--bang regime through boundary saturation, a phenomenon that emerges naturally from constrained optimization and disappears once interior optimal controls are restored.

Sensitivity analyses show that equilibrium properties are primarily governed by control and inventory costs, while variations in action discretization or market depth mainly rescale equilibrium magnitudes without altering qualitative stability. In particular, increasing pool size dampens individual impact but does not destabilize the equilibrium, whereas stronger regularization suppresses extreme trading behavior. Moreover, the terminal inventory target $x^*$ governs the equilibrium direction itself: switching from a liquidation target to an accumulation target reverses the aggregate flow and the price trajectory, demonstrating that the model's qualitative predictions are directly tied to the economic incentives faced by the traders.

Overall, the proposed framework yields stable, interpretable equilibria and clarifies the regimes in which mean field approximations remain reliable for liquidity pool dynamics. Beyond validating the theoretical construction, the numerical results provide a quantitative map of equilibrium behavior in AMMs and underline both the robustness and the limitations of applying classical MFG models to decentralized markets. 

This work represents a first step toward a comprehensive modeling framework for AMM-based markets. One natural extension of this research is the inclusion of transaction costs in the model, an important factor in real-world DeFi protocols. Accounting for trading fees and their impact on agents' strategies is nontrivial and currently under investigation. These extensions aim to provide a more complete and realistic representation of decentralized liquidity markets, opening avenues for future theoretical and empirical studies. Moreover, we plan to enrich the model by incorporating additional economic agents, such as liquidity providers and arbitrageurs, whose strategic behavior is essential to capture the full dynamics of decentralized exchanges.

\section*{Statements and Declarations}

\textbf{Funding.} This work has received funding from the European Union’s Horizon 2020 research and innovation programme under the Marie Skłodowska-Curie, H2020-MSCA-RISE-2017 Project No. 777822,
"Geometric and Harmonic Analysis with Interdisciplinary Applications".

\textbf{Competing Interests.} The authors have no relevant financial or non-financial interests to disclose.

\textbf{Data Availability.} The simulation code and data supporting the findings of this study are available from the corresponding author upon reasonable request.

\textbf{Author Contributions.} All authors contributed to the study conception, model design, and manuscript preparation. The numerical implementation was developed by A. Mu\~noz Gonz\'alez. All authors reviewed and approved the final manuscript.

\appendix
\section{Algorithmic Implementation}
\label{sec:algorithm}

This appendix summarizes the numerical machinery used to compute the MFG
equilibrium flow and to evaluate unilateral deviations.  
The objective is not to describe software details, but to document the
mathematical structure of the fixed–point iteration, the best–response
computation, and the ``frozen policy’’ mechanism used in the $\epsilon$–Nash test.
All notation follows the model formulation in Section~\ref{sec:Model Formulation}.

\subsection*{Simulation loop (agent dynamics).}
For a fixed aggregate path $q(\cdot)$, the simulator evolves the pool and the agents on a time grid $(t_k)_{k=0}^{K}$ as follows:
\begin{enumerate}
    \item each agent selects an action $\alpha_i(t_k)$ according to the chosen policy;
    \item the pool state is updated using the imposed aggregate flow $q(t_k)$;
    \item each agent updates its inventory state $(X_i,Y_i)$ given its own action $\alpha_i(t_k)$ and the current pool price.
\end{enumerate}
The output of one run is the full control history $(\alpha_i(t_k))_{i,k}$ together with the induced mean control $\bar{\alpha}(t_k)=\frac{1}{N}\sum_i \alpha_i(t_k)$, which defines the operator $\mathcal{F}$ used in the fixed point solver.

\subsection*{State variables and controls}

Each trader is characterized by an inventory--cash pair
\[
X^i_t \in \mathbb{R}_+,\qquad Y^i_t \in \mathbb{R},
\]
evolving according to the AMM swap rule (constant–product pool):
\[
x_{t+\Delta t}
    = \max(\varepsilon, x_t - \alpha^i_t\,\Delta t),\qquad
y_{t+\Delta t}
    = \frac{k}{x_{t+\Delta t}},
\]
with price
\(
P_t = k/x_t^2.
\)

The individual control is a bounded liquidation rate
\[
\alpha^i_t \in A \subset \mathbb{R},
\]
where $A$ is a finite grid used for discrete optimization.

The running and terminal costs used in the numerical experiments are:
\[
h(x) = \lambda x^2, \qquad
\ell(x) = \eta_{\mathrm{term}}(x - x^\star)^2,
\qquad 
c(a) = \rho a^2.
\]

\subsection*{Mean field iteration}

Given an aggregate flow path $q(\cdot)$, we define $\mathcal{F}(q)$ as the empirical mean of individual best responses when the AMM is driven by $q$:
\[
\mathcal{F}(q)(t)\;:=\;\frac{1}{N}\sum_{i=1}^N \alpha_i^{\,q}(t),
\]
where $\alpha_i^{\,q}$ denotes the control produced by agent $i$ when the aggregate impact is fixed to $q$ (treated as exogenous).
The mean field equilibrium $q^\star$ is a consistency condition,
\[
q^\star=\mathcal{F}(q^\star),
\]
i.e., the aggregate flow used to price the pool coincides with the average flow generated by the population's best responses.
For a generic input path $q$, the output $\mathcal{F}(q)$ differs from $q$, and the fixed point iteration searches for a self--consistent trajectory.

We compute the mean-field equilibrium as a fixed point of the best-response operator
$q \mapsto \mathcal{F}(q)$.
We initialize with the zero path $q^{(0)}(t)\equiv 0$, which provides a neutral baseline and avoids injecting directional bias into the fixed-point iteration. We iterate
\[
q^{(k+1)} \;=\; (1-\omega)\, q^{(k)} \;+\; \omega\, \mathcal{F}\bigl(q^{(k)}\bigr),
\qquad \omega \in (0,1],
\]
until the stopping criterion
\[
\|q^{(k+1)}-q^{(k)}\| \le \texttt{tol}
\]
is met (or a maximum number of iterations is reached).
Relaxation ($\omega<1$) improves numerical stability when $\mathcal{F}$ is only weakly contractive.

\subsection*{Best response computation}
\label{subsec:bestresponse}

Given a candidate population flow $q(t)$, each trader solves a 
myopic quadratic surrogate of the one–period optimization problem.
For each action $a\in A$ we compute an approximation of the
instantaneous value difference:
\[
Q(a) \approx
\big(Y^i + X^i P_t\big)_{\text{post-swap}(a)}
-
\big(Y^i + X^i P_t\big)_{\text{pre-swap}}
\;-\;
h(X^i_t)\,\Delta t
\;-\;
c(a) \Delta t
\;-\;
\ell'(X^i_t)\,a\,\Delta t,
\]
where the last term is a one–step linearization of the terminal cost.

The trader's policy is then an argmax rule:
\[
\alpha^i_t = \arg\max_{a\in A} Q(a).
\]

The simulation of all traders under this policy yields $\mathcal{F}(q)$.

\subsection*{Frozen–policy mechanism ($\epsilon$–Nash test)}
\label{subsec:frozenpolicy}

To test the unilateral optimality of the equilibrium $q^\star$,  
we store the baseline actions of each agent during the equilibrium 
simulation:
\[
\alpha^{i,\star}_t,\qquad 0\le t\le T.
\]

To evaluate whether trader $i$ can gain by deviating, we re-run the system 
with:
\begin{itemize}
    \item \textbf{Trader $i$ free:} uses a modified policy $\pi_{i}^{\mathrm{dev}}$ obtained by scanning a grid of inventory targets $x^*$. 
    \item \textbf{All other traders $j\neq i$ frozen:} they replay exactly their baseline actions $\alpha^{j,\star}_t$.
\end{itemize}

We denote this composite policy by 
\[
\pi^{\mathrm{frozen}}(j,t) =
\begin{cases}
\alpha^{j,\star}_t, & j\neq i,\\
\pi_{i}^{\mathrm{dev}}(t), & j=i.
\end{cases}
\]

Running the simulator under this policy yields a new objective $J_i^{\mathrm{dev}}$.
The unilateral improvement is
\[
\Delta J_i = J_i^{\mathrm{dev}} - J_i^\star.
\]
\subsection*{Pseudocode summary}

\begin{algorithm}[H]
\caption{Fixed–point solver}
\begin{algorithmic}[1]
\State Initialize $q^{(0)}(t)=0$
\For{$k=0,1,2,\dots$}
    \State Run simulator with external flow $q^{(k)}$ to obtain 
           trader actions $\alpha^{i,(k)}_t$
    \State Compute mean flow 
           $\bar\alpha^{(k)}(t)=\frac1N\sum_i\alpha^{i,(k)}_t$
    \State Update $q^{(k+1)} = \text{Acceleration}(q^{(k)},\bar\alpha^{(k)})$
    \If{$\|q^{(k+1)}-q^{(k)}\|_\infty < \text{tol}$}
        \State \Return $q^\star = q^{(k+1)}$
    \EndIf
\EndFor
\end{algorithmic}
\end{algorithm}

\begin{algorithm}[H]
\caption{$\epsilon$–Nash unilateral deviation test}
\begin{algorithmic}[1]
\State Run equilibrium simulation $\Rightarrow \alpha^{i,\star}_t$, $J_i^\star$
\For{agent $i \in \{1,\dots,N\}$ or a random subset}
    \For{inventory target $x^\star$ in grid}
        \State Construct frozen policy:
        \[
        \pi(j,t)=
        \begin{cases}
        \alpha^{j,\star}_t, & j\neq i,\\
        \pi^{\mathrm{dev}}_{i}(t;x^\star), & j=i
        \end{cases}
        \]
        \State Run simulator $\Rightarrow J_i^{\mathrm{dev}}$
        \State Record $\Delta J_i = J_i^{\mathrm{dev}} - J_i^\star$
    \EndFor
\EndFor
\State Return histogram of $\Delta J_i$
\end{algorithmic}
\end{algorithm}

\subsection*{Parameter summary}

The following parameters are varied across experiments:

\begin{itemize}
    \item \textbf{Population size $N$:} number of agents used to approximate the mean field.
    \item \textbf{Time discretization:} number of steps and $dt$.
    \item \textbf{Action grid $A$:} size and resolution of admissible controls.
    \item \textbf{Market depth:} $(X_0,Y_0)$ initial liquidity of the AMM.
    \item \textbf{Noise:} $\sigma_i$ idiosyncratic heterogeneity.
    \item \textbf{Initial inventory distribution $(x_0^\mu, x_0^\sigma)$:} mean and dispersion of agents’ initial states.
    \item \textbf{Costs:} $(\lambda, \eta_{\mathrm{term}}, \rho)$ controlling running, terminal, and control penalties.
    \item \textbf{Inventory target $x^\star$:} reference level toward which agents may be pulled by terminal incentives.
    \item \textbf{Solver parameters:} relaxation rate, Anderson window, temperature, tolerance.
\end{itemize}

This modular structure allows one to isolate the effect of each modeling component and yields reproducible numerical comparisons.

\clearpage
\section{Figures and Tables}

\begin{figure}[!htbp]
    \centering
    \includegraphics[width=0.85\textwidth]{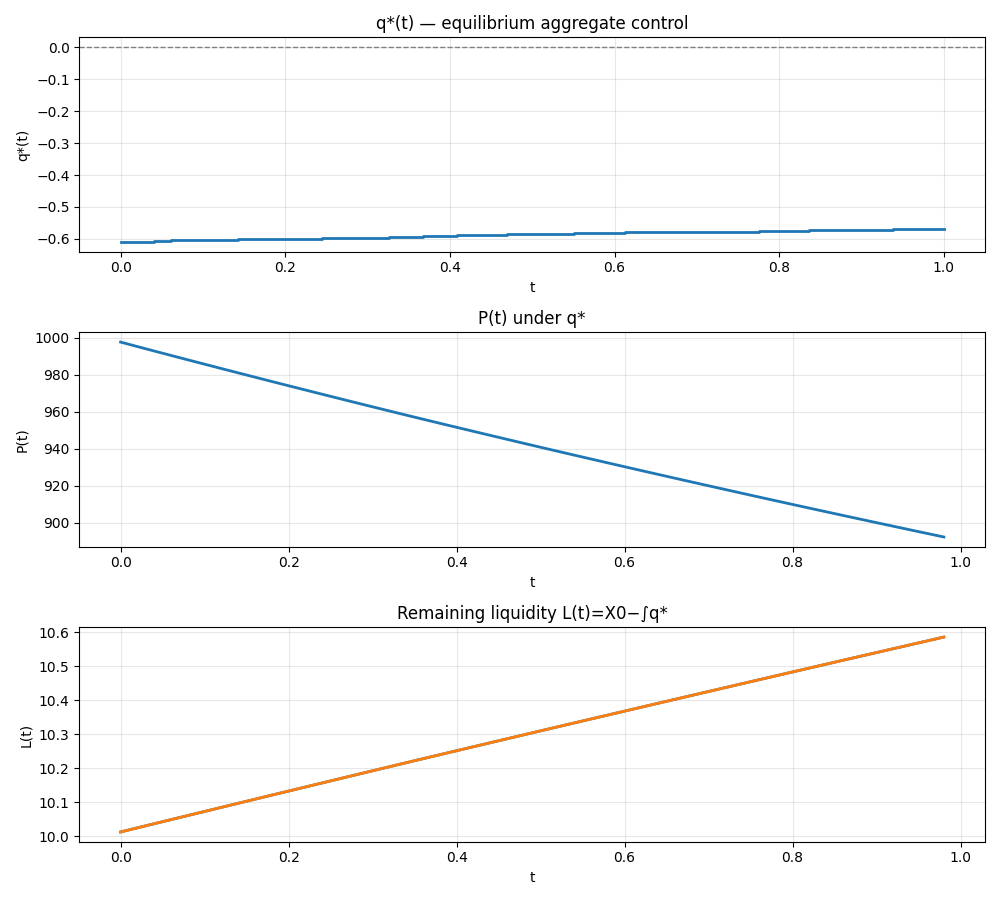}
    \caption{Mean field equilibrium of the baseline scenario ($\lambda=2$, $\eta=0.3$, $\rho=5$, $A=[-1.5, 1.5]$).
    \textbf{Top panel:} aggregate equilibrium control $q^*(t)$, strictly negative for all $t$, indicating sustained ETH selling.
    \textbf{Middle panel:} price trajectory $P(t)$ induced by $q^*$, strictly decreasing from $P_0 \approx 998$ to $P_T \approx 892$.
    \textbf{Bottom panel:} remaining liquidity $L(t) = X_0 - \int_0^t q^*(s)\,ds$, which remains bounded below ($L(t) \geq 10.01$), confirming no pool depletion.}
    \label{fig:qstar_precio_baseline}
\end{figure}

\begin{figure}[!htbp]
    \centering
    \includegraphics[width=0.85\textwidth]{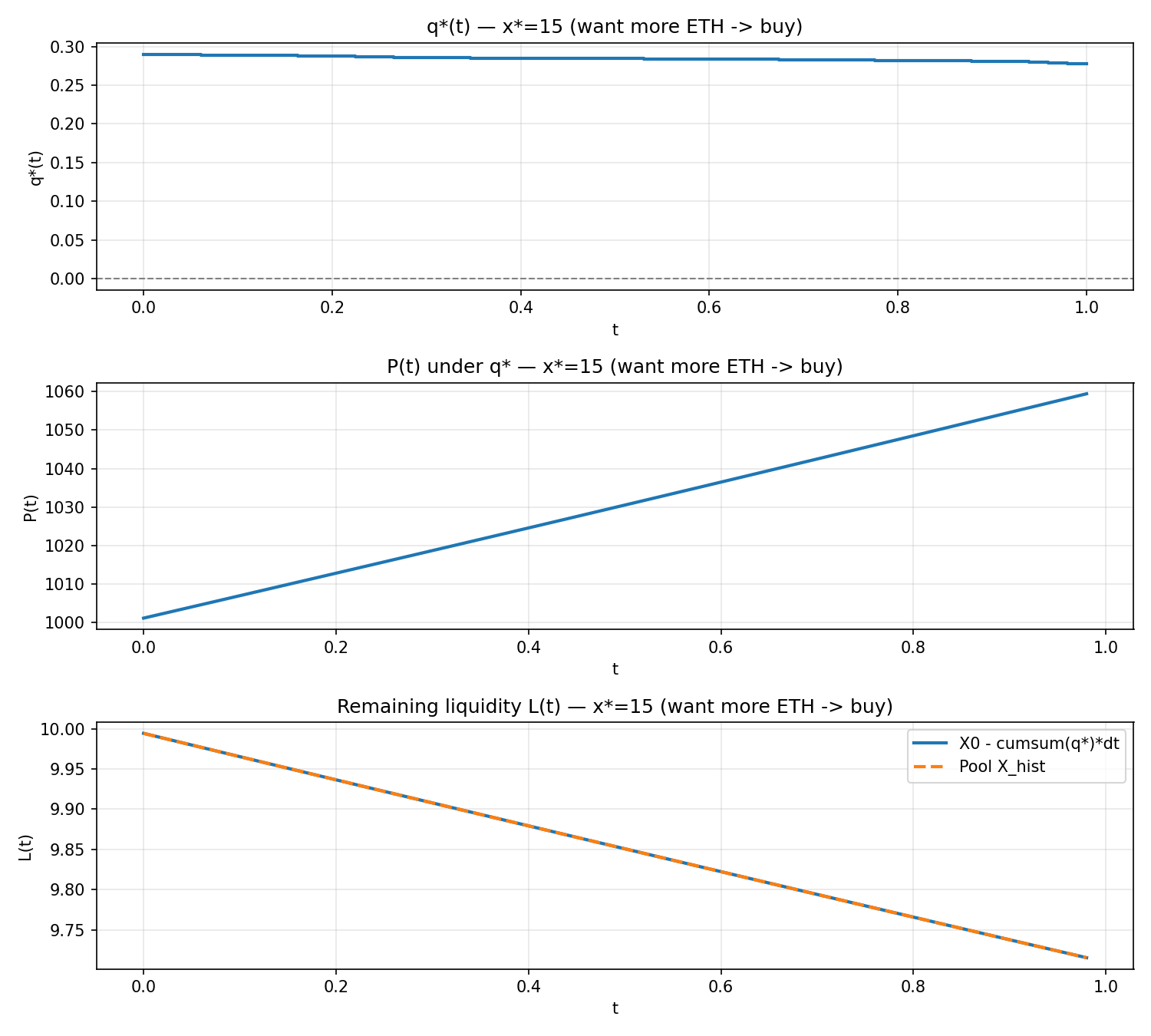}
    \caption{Mean field equilibrium with incentive to accumulate ETH ($x^* = 15 > x_0 = 10$, $\lambda=2$, $\eta=0.3$, $\rho=5$). Compare with Figure~\ref{fig:qstar_precio_baseline}. \textbf{Top panel:} $q^*(t) > 0$ for all $t$, indicating sustained buying. \textbf{Middle panel:} $P(t)$ strictly increasing, from $P_0 \approx 1001$ to $P_T \approx 1059$. \textbf{Bottom panel:} remaining liquidity $L(t)$ decreasing, reflecting the extraction of ETH from the pool by traders.}
    \label{fig:qstar_mirror_xstar}
\end{figure}

\begin{figure}[!htbp]
    \centering
    \includegraphics[width=0.9\linewidth]{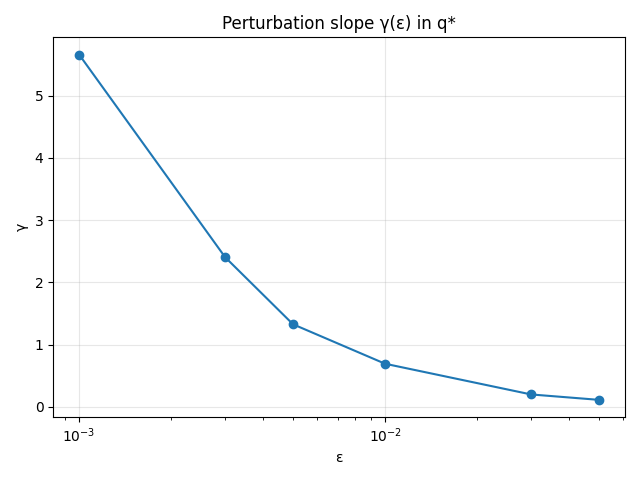}
    \caption{
    Local perturbation sensitivity at the computed fixed point control $q^\star$.
    We report the empirical slope $\gamma(\varepsilon)$ obtained by re-solving the best-response problem under a small perturbation of the mean-field input around $q^\star$.
    Smaller values indicate a more locally contractive response map and hence a more stable fixed point in practice.
    }
    \label{fig:perturb}
\end{figure}

\begin{figure}[!htbp]
    \centering
    \includegraphics[width=0.9\linewidth]{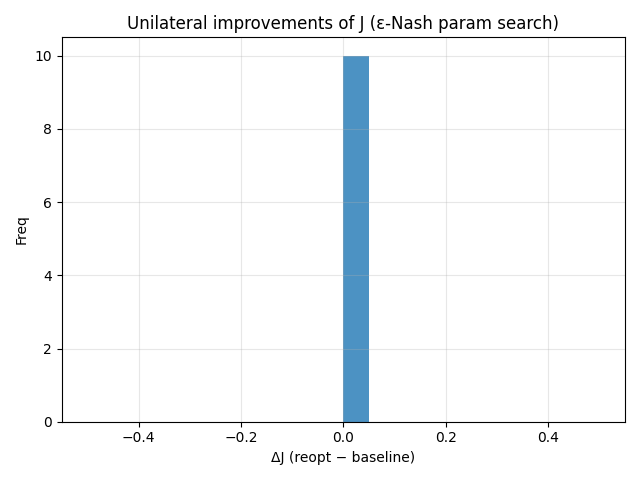}
    \caption{
    Unilateral deviation check at $q^\star$ (frozen-policy $\varepsilon$-Nash test).
    For each agent $i$, we keep all other agents on their baseline trajectories and re-optimize only agent $i$; the histogram shows $\Delta J_i := J_i(\text{deviate})-J_i(\text{baseline})$.
    Values concentrated near $0$ indicate no profitable unilateral deviations at the tested resolution.
    }
    \label{fig:unilateral}
\end{figure}

\begin{figure}[!htbp]
    \centering
    \includegraphics[width=0.9\linewidth]{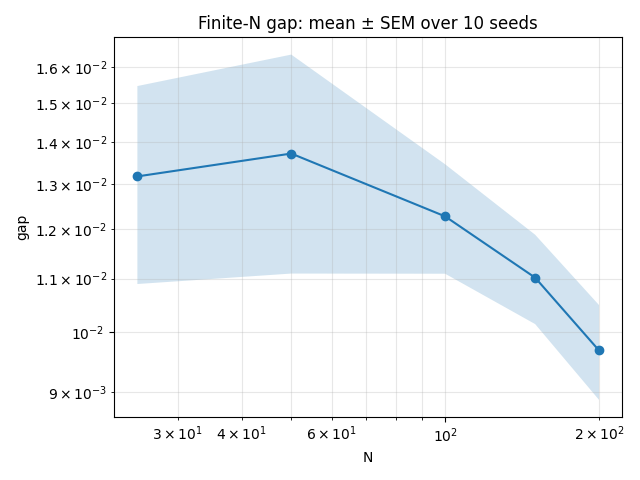}
    \caption{
    Finite-$N$ gap between the empirical mean control $\bar{\alpha}^N$ and the mean-field fixed point $q^\star$.
    For each $N$ we report the mean gap over multiple random seeds together with an uncertainty band (mean $\pm$ SEM).
    A decreasing trend supports convergence of the finite-player system toward the mean-field solution.
    }
    \label{fig:finiten}
\end{figure}

\clearpage
\begin{obs*}
    All numerical metrics exhibit mild run-to-run variability due to stochastic initialization, finite discretization, and numerical tolerances.
    Reported values should therefore be interpreted as representative magnitudes rather than exact quantities.
    All qualitative conclusions and comparative trends reported in this section are robust across independent runs.    
\end{obs*}

\begin{table}[H]
\centering
\caption{Baseline scenario: equilibrium magnitude, stability indicators, and finite-player deviations across action set sizes.}
\label{tab:baseline}
\begin{tabular}{lcccc}
\hline
Config & 
$\|q^\ast\|_2$ &
$\overline{\gamma}$ &
gap$(N_{\max})$ &
sat\_rate
\\
\hline
$A=[-1.5,1.5]$ & $0.59$ & $1.65$ & $8.5457e-03$ & $0.0$ \\
\hline
\end{tabular}
\end{table}

\begin{table}[H]
\centering
\caption{Experiment 0 (frictionless degeneracy): equilibrium selection depends on admissible set and numerical asymmetries.}
\label{tab:exp0}
\begin{tabular}{lcccc}
\toprule
Action set $A$ & $\mathrm{mid}(A)$ & $\|q^*\|_2$ & $\overline{\gamma}$ & $gap(N_{\max})$ \\
\midrule
$[-1.5,1.5]$    & $0$   & $0.10$ & $365.5$     & $8.97\times 10^{-1}$ \\
$[-3,3]$        & $0$   & $0.20$ & $744.4$     & $1.03\times 10^{0}$ \\
$[1,3]$         & $2$   & $2.04$ & $948.8$ & $3.22\times 10^{-2}$ \\
\bottomrule
\end{tabular}
\end{table}

\begin{table}[H]
\centering
\caption{Experiment A: Effect of the action set.}
\label{tab:expA}
\begin{tabular}{lcccc}
\toprule
Config & 
$\|q^\ast\|_2$ &
$\overline{\gamma}$ &
gap$(N_{\max})$ &
sat\_rate
\\
\midrule
$A=[-0.5,0.5]$ & $0.5$ & $0.20$ & $8.0681e-04$  & $1.0$ \\
$A=[-1.5,1.5]$ & $0.59$ & $1.65$ & $8.5457e-03$ & $0.0$ \\
$A=[-10,10]$ & $0.66$ & $3.36$ & $6.1986e-03$  & $0.0$ \\
\bottomrule
\end{tabular}%
\end{table}

\begin{table}[H]
\centering
\caption{Experiment B: Market depth (pool size).}
\label{tab:expB}
\begin{tabular}{lcc}
\toprule
Config &
$\min_t L(t)$ &
gap$(N_{\max})$
\\
\midrule
$X_0=10, Y_0=10000.0$ & $10.01$ & $8.5457e-03$  \\
$X_0=100, Y_0=100000.0$ & $100.01$ & $8.5457e-03$ \\
\bottomrule
\end{tabular}
\end{table}

\begin{table}[H]
\centering
\caption{Experiment C: Effect of cost parameters.}
\label{tab:expC}
\begin{tabular}{lccc}
\toprule
Config &
$\|q^\ast\|_2$ &
$\overline{\gamma}$ &
sat\_rate
\\
\midrule
$\lambda,\eta,\rho=2, 0.3, 5$ & $0.59$ & $1.65$ & $0.0$ \\
$\lambda,\eta,\rho=2, 0.3, 0.5$ & $1.5$ & $0.0$ & $1.0$\\
$\lambda,\eta,\rho=2, 3, 5$ & $1.5$ & $0.0$ & $1.0$\\
$\lambda,\eta,\rho=10, 0.3, 5$ & $0.59$ & $1.65$ & $0.0$\\
\bottomrule
\end{tabular}%
\end{table}

\begin{table}[H]
\centering
\caption{Experiment D: effect of the terminal target $x^*$ on the equilibrium.}
\label{tab:exp_D}
\begin{tabular}{lccc}
\hline
Config & $\|q^*\|_2$ & $q^*$ sign & $P_0 \to P_T$ \\
\hline
$x^* = 0$ (baseline) & $0.59$ & $< 0$ (sell) & $998 \to 892$ \\
$x^* = 15$ & $0.28$ & $> 0$ (buy) & $1001 \to 1059$ \\
\hline
\end{tabular}
\end{table}

\textbf{Notes.}
$\overline{\gamma}$ denotes the mean perturbation slope averaged over $\epsilon\in\{10^{-3}, 3*10^{-3}, 5*10^{-3}, 10^{-2}, 3*10^{-2}, 5*10^{-2}\}$; see Section \ref{sec:baseline}, Stability to perturbations. 

|$\|q^*\|_2$ is the discrete $L^2$ norm of the equilibrium flow trajectory over $[0,T]$; since $q^*$ is a time-dependent path, this norm captures its aggregate magnitude across the full horizon, complementing the pointwise evolution shown in the Figures \ref{fig:qstar_precio_baseline} and \ref{fig:qstar_mirror_xstar}.

\bibliographystyle{plainnat}
\bibliography{main}
\end{document}